\documentclass[12pt]{article}
\usepackage{amsmath,amssymb,amsthm}
\numberwithin{equation}{section}
\usepackage{units}
\usepackage{color}
\usepackage[T1]{fontenc}
\usepackage[utf8]{inputenc}
\usepackage{authblk}
\usepackage{bm}
\usepackage{graphicx}

\usepackage{datetime}

\usepackage[colorlinks=true,
linkcolor=webgreen,
filecolor=webbrown,
citecolor=webgreen]{hyperref}

\definecolor{webgreen}{rgb}{0,.5,0}
\definecolor{webbrown}{rgb}{.6,0,0}

\usepackage[toc,page]{appendix}

\newtheorem{thm}{Theorem}
\newtheorem{theorem}[thm]{Theorem}
\newtheorem{lemma}{Lemma}

\newtheorem{corollary}[thm]{Corollary}

\title{Basic properties of a generalized third order sequence of numbers}
\author[]{Kunle Adegoke \\\href{mailto:kunle.adegoke@yandex.com}{\tt kunle.adegoke@yandex.com}}

\affil{Department of Physics and Engineering Physics, \mbox{Obafemi Awolowo University}, 220005 Ile-Ife, Nigeria}

\begin{document}
\date{}

\maketitle

\begin{abstract} 
\noindent We study the properties of the third order sequence $(w_n)=\left(w_n(a,b,c; r, s,t)\right)$ defined by the recurrence relation $w_n  = rw_{n - 1}  + sw_{n - 2} + tw_{n - 3}\, (n \ge 3)$ with $w_0  = a,\,w_1  = b,\,w_2=c$, where $a$, $b$, $c$, $r$, $s$ and $t$ are arbitrary complex numbers and $t\ne 0$. Properties examined include the partial sum of the terms of the sequence, with indices in arithmetic progression, as well as double binomial summation identities.
\end{abstract}
\section{Introduction and preliminary results}
We wish to study the properties of the linear homogeneous third order sequence $(w_n)=\left(w_n(a,b,c; r, s,t)\right)$ defined by the recurrence relation
\begin{equation}\label{eq.vhrb5b3}
w_n  = rw_{n - 1}  + sw_{n - 2} + tw_{n - 3}\, (n \ge 3)\,;
\end{equation}
with $w_0  = a,\,w_1  = b,\,w_2=c$, where $a$, $b$, $c$, $r$, $s$ and $t$ are arbitrary complex numbers, with $t\ne 0$. 

\medskip

The sequence $(w_n)$ and its special cases have been studied by various authors, notably Jarden \cite{jarden66}, Shannon and Horadam \cite{shannon72}, Yalavigi \cite{yalavigi72}, Pethe \cite{pethe88}, Gerdes \cite{gerdes78}, Waddil \cite{waddill91} and Rabinowitz \cite{rabinowitz96}. Our main aim here is to develop presumably new properties of $(w_n)$. Chief among these are the partial sum of the terms of the sequence, with indices in arithmetic progression, double binomial summation identities and multiple argument formulas.

\medskip

The table below shows some special cases of $(w_n)$, the first four of which are now well known in the literature.
\begin{table}[h!]
\begin{tabular}{lllllll}
Name & $a$ & $b$ & $c$ & $r$ & $s$ & $t$ \\ 
\hline
Tribonacci, $T_n$ & $0$ & 1 & 1 & 1 & 1 & 1 \\ 
Tribonacci-Lucas, $K_n$ & 3 & 1 & 3 & 1 & 1 & 1 \\ 
Padovan, $P_n$ & 1 & 1 & 1 & 0 & 1 & 1 \\ 
Perrin, $Q_n$ & 3 & 0 & 2 & 0 & 1 & 1 \\ 
Generalized Tribonacci, $\mathcal{T}_n$ & $a$ & $b$ & $c$ & 1 & 1 & 1 \\ 
Generalized Padovan, $\mathcal{P}_n$ & $a$ & $b$ & $c$ & 0 & 1 & 1 \\ 
\end{tabular}
\end{table}

Let $\alpha$, $\beta$ and $\gamma$ be the zeroes (assumed all distinct) of the characteristic polynomial \mbox{$x^3-rx^2-sx-t$} of the sequence $(w_n)$. Then,
\begin{equation}\label{eq.g4j8ymn}
\alpha  + \beta  + \gamma  = r,\;\alpha \beta \gamma  = t,\;\alpha \beta  + \alpha \gamma  + \beta \gamma  =  - s\,.
\end{equation}
Standard methods for solving difference equations give
\begin{equation}
w_n=A\alpha^n+B\beta^n+C\gamma^n\,,
\end{equation}
where
\begin{equation}
A = \frac{{\beta (a\gamma  - b) + c - b\gamma }}{{(\alpha  - \beta )(\alpha  - \gamma )}},\;B = \frac{{\gamma (a\alpha  - b) + c - b\alpha }}{{(\beta  - \alpha )(\beta  - \gamma )}},\;C = \frac{{\alpha (a\beta  - b) + c - b\beta }}{{(\gamma  - \alpha )(\gamma  - \beta )}}\,.
\end{equation}
Three important particular cases of $(w_n)$ are $(u_n(r,s,t))$, $(v_n(r,s,t))$ and $(z_n(r,s,t))$, with terms given by $u_n(r,s,t)=w_n(0,1,r;r,s,t)$, $v_n(r,s,t)=w_n(3,r,r^2+2s;r,s,t)$ and $z_n(r,s,t)=w_n(1,1,1;r,s,t)$. Thus,
\begin{equation}
u_n=A'\alpha^n+B'\beta^n+C'\gamma^n\,,
\end{equation}
\begin{equation}\label{eq.o6c7482}
v_n=\alpha^n+\beta^n+\gamma^n
\end{equation}
and
\begin{equation}
z_n=A''\alpha^n+B''\beta^n+C''\gamma^n\,;
\end{equation}
where
\begin{equation}
A' = \frac{\alpha}{{(\alpha  - \beta )(\alpha  - \gamma )}},\;B' = \frac{\beta}{{(\beta  - \alpha )(\beta  - \gamma )}},\;C' = \frac{\gamma}{{(\gamma  - \alpha )(\gamma  - \beta )}}\,,
\end{equation}
and
\begin{equation}
A'' = \frac{(\gamma-1)(\beta-1)}{{(\alpha  - \beta )(\alpha  - \gamma )}},\;B'' = \frac{(\gamma-1)(\alpha-1)}{{(\beta  - \alpha )(\beta  - \gamma )}},\;C'' = \frac{(\alpha-1)(\beta-1)}{{(\gamma  - \alpha )(\gamma  - \beta )}}\,.
\end{equation}
Note that $T_n=u_n(1,1,1)$, $K_n=v_n(1,1,1)$, $P_n=z_n(0,1,1)$ and $Q_n=v_n(0,1,1)$.

\medskip

Extension of the definition of $(w_n)$ to negative indices is provided by writing the recurrence relation as
\begin{equation}\label{eq.boxzh6a}
w_{-n}=(w_{-(n-3)}-rw_{-(n-2)}-sw_{-(n-1)})/t\,.
\end{equation}
The table below shows the first few terms of the sequences $(u_n)$, $v_n$ and $(z_n)$:
\begin{table}[h!]
\begin{tabular}{cccccccl}
$n$ & $-2$ & $-1$ & $0$ & $1$ & $2$ & $3$ & \multicolumn{1}{c}{$4$} \\ 
\hline
$u_n$ & $1/t$ & $ 0$ & $ 0$ & $ 1$ & $ r$ & $ r^2+s$ & \multicolumn{1}{c}{$ r^3+2rs+t$} \\ 
$v_n$ & $(s^2-2rt)/t^2$ & $ -s/t$ & $ 3$ & $ r$ & $ r^2+2s$ & $ r^3+3rs+3t$ & \multicolumn{1}{c}{$ r^4+4sr^2+4rt+2s^2$} \\ 
$z_n$ & $((t-s)(1-r)+s^2)/t^2$ & $ (1-r-s)/t$ & $ 1$ & $ 1$ & $ 1$ & $ t+r+s$ & \multicolumn{1}{c}{$ (t+s)(1+r)+r^2$} \\ 
\end{tabular}
\end{table}

The identities in \eqref{eq.g4j8ymn} and identity \eqref{eq.o6c7482} can be collected as
\begin{equation}
\alpha ^n  + \beta ^n  + \gamma ^n  = v_n\,,
\end{equation}
\begin{equation}
(\alpha \beta \gamma )^n  = t^n
\end{equation}
and
\begin{equation}
(\alpha \beta )^n  + (\alpha \gamma )^n  + (\beta \gamma )^n  = \frac{{(\alpha \beta \gamma )^n }}{{\gamma ^n }} + \frac{{(\alpha \gamma \beta )^n }}{{\beta ^n }} + \frac{{(\beta \gamma \alpha )^n }}{{\alpha ^n }} = t^n v_{ - n}\,.
\end{equation}
\begin{theorem}\label{theorem.jbqvo05}
The following identity holds for integers $n$ and $m$:
\[
w_{n+m}=u_nw_{m+1}+(u_{n+1}-ru_{n})w_m+tu_{n-1}w_{m-1}\,.
\]
\end{theorem}
In particular, we have
\begin{equation}\label{eq.yn0hmyp}
w_n=bu_n+(c-br)u_{n-1}+atu_{n-2}
\end{equation}
which also gives
\begin{equation}
v_n=ru_n+2su_{n-1}+3tu_{n-2}\,,
\end{equation}
\begin{equation}
z_n  = u_n  + (1 - r)u_{n - 1}  + tu_{n - 2}
\end{equation}
and
\begin{equation}\label{eq.txorvr0}
\alpha^n=u_{n-1}\alpha^2+(u_{n}-ru_{n-1})\alpha+tu_{n-2}\,;
\end{equation}
with similar expressions for $\beta^n$ and $\gamma^n$.
\begin{proof}
We seek to express a number from the $(w_n)$ sequence as a linear combination of three numbers from the $(u_n)$ sequence. Let 
\begin{equation}\label{eq.nbl33qg}
w_{m+n}=f_1u_n+f_2u_{n-1}+f_3u_{n-2}\,,
\end{equation}
where the coefficients $f_1$, $f_2$ and $f_3$ are to be determined. Setting $n=0$, $n=1$ and $n=2$, in turn, we find $f_1=w_{m+1}$, $f_2=w_{m+2}-rw_{m+1}$ and $f_3=tw_m$ and the identity of the theorem is established after shifting the indices $m$ and $n$.
\end{proof}
Note that identity\eqref{eq.txorvr0} follows from the fact that the sequences $(\alpha^n)$, $(\beta^n)$ and $(\gamma^n)$ are also special cases of the generalized sequence $(w_n)$.

\medskip

The identities below, in Theorem \ref{theorem.f1jq8l2}, based on identities \eqref{eq.o6c7482} and \eqref{eq.txorvr0}, make it possible to access the negative index terms of $(w_n)$ and the special cases directly, without using the recurrence relation \eqref{eq.boxzh6a}.
\begin{lemma}\label{lemma.m2hum6l}
Let $a_0$, $a_1$, $a_2$, $\cdots$, $a_n$ and $b_0$, $b_1$, $b_2$, $\cdots$, $b_n$ be rational numbers. Let $\lambda_1$, $\lambda_2$, $\cdots$, $\lambda_n$ be linearly independent irrational numbers. Then
\[
a_0  + a_1 \lambda _1  + a_2 \lambda _2  +  \cdots  + a_n \lambda _n  = b_0  + b_1 \lambda _1  + b_2 \lambda _2  +  \cdots  + b_n \lambda _n 
\]
implies that $a_0=b_0$, $a_1=b_1$, $a_2=b_2$, $\cdots$, $a_n=b_n$.
\end{lemma}
\begin{theorem}\label{theorem.f1jq8l2}
The following identities hold for any integer $n$:
\begin{equation}\label{eq.lmwxb7q}
u_{ - n}  =(u_{n - 1}^2  - u_n u_{n - 2} )/t^{n-1}\,,
\end{equation}
\begin{equation}\label{eq.xjwme1i}
v_{-n}=(v_n^2-v_{2n})/(2t^n)\,.
\end{equation}
\end{theorem}
\begin{proof}
Since $\alpha$ and $\alpha^2$ are linearly independent irrationals, write
\begin{equation}\label{eq.x59evau}
\alpha^{-n}=c_1\alpha^2+c_2\alpha+c_3\,,
\end{equation}
where $n$ is an integer and $c_1$, $c_2$ and $c_3$ are to be determined.
Thus, we have
\begin{equation}\label{eq.xircjfm}
\begin{split}
1 &= c_1 \alpha ^{n + 2}  + c_2 \alpha ^{n + 1}  + c_3 \alpha ^n\\
&= \alpha ^2 (c_1 f_{n + 2}  + c_2 f_{n + 1}  + c_3 f_n ) + \alpha (c_1 g_{n + 2}  + c_2 g_{n + 1}  + c_3 g_n )\\
&\quad+ (c_1 h_{n + 2}  + c_2 h_{n + 1}  + c_3 h_n )\,,
\end{split}
\end{equation}
where we have used identity \eqref{eq.txorvr0} with $f_n=u_{n-1}$, $g_n=u_n-ru_{n-1}$ and $h_n=tu_{n-2}$.

\medskip

From Lemma \ref{lemma.m2hum6l} and identity \eqref{eq.xircjfm} we have
\[
f_{n + 2} c_1  + f_{n + 1} c_2  + f_n c_3  = 0\,,
\]
\[
g_{n + 2} c_1  + g_{n + 1} c_2  + g_n c_3  = 0
\]
and
\[
h_{n + 2} c_1  + h_{n + 1} c_2  + h_n c_3  = 1\,,
\]
to be solved simultaneously for the coefficients $c_1$, $c_2$ and $c_3$.

\medskip

We find
\begin{equation}\label{eq.n7r35m0}
c_1=\frac{g_nf_{n+1}-f_ng_{n+1}}{\Delta_n}\,,
\end{equation}
where
\begin{equation}
\Delta _n  = \left| {\begin{array}{*{20}c}
   {f_{n + 2} } & {f_{n + 1} } & {f_n }  \\
   {g_{n + 2} } & {g_{n + 1} } & {g_n }  \\
   {h_{n + 2} } & {h_{n + 1} } & {h_n }  \\
\end{array}} \right|\,.
\end{equation}
Now, from \eqref{eq.txorvr0}, we have
\begin{equation}\label{eq.m0bph93}
\alpha^{-n}=f_{-n}\alpha^2+g_{-n}\alpha+h_{-n}\,.
\end{equation}
Equating coefficients of $\alpha^2$ from \eqref{eq.x59evau} and \eqref{eq.m0bph93}, using \eqref{eq.n7r35m0}, and shifting index gives
\begin{equation}
f_{ - n+1}  = \frac{{g_{n-1} f_{n}  - f_{n-1} g_{n} }}{{\Delta _{n-1} }}\,,
\end{equation}
so that we have
\begin{equation}
u_{ - n}  = \frac{{u_{n - 1}^2  - u_n u_{n - 2} }}{{\Delta _{n - 1} }}\,.
\end{equation}
Application of the principle of mathematical induction shows that $\Delta_n=t^n$ and identity \eqref{eq.lmwxb7q} follows.

\medskip

We have
\[
\begin{split}
v_n^2  &= (\alpha ^n  + \beta ^n  + \gamma ^n )^2\\
 &= \alpha ^{2n}  + \beta ^{2n}  + \gamma ^{2n}  + 2((\alpha \beta )^n  + (\alpha \gamma )^n  + (\beta \gamma )^n )\\
 &= \alpha ^{2n}  + \beta ^{2n}  + \gamma ^{2n}  + 2t^n (\gamma ^{ - n}  + \beta ^{ - n}  + \alpha ^{ - n} )\\
 &= v_{2n}  + 2t^n v_{ - n}\,,
\end{split}
\]
from which identity \eqref{eq.xjwme1i} follows.
\end{proof}
Use of \eqref{eq.lmwxb7q} in \eqref{eq.yn0hmyp} gives
\begin{equation}\label{eq.fsxsk5o}
w_{ - n}  = \frac{{bt(u_{n - 1}^2  - u_n u_{n - 2} ) + (c - br)(u_n^2  - u_{n + 1} u_{n - 1} ) + a(u_{n + 1}^2  - u_{n + 2} u_n )}}{{\left( {bu_n  + (c - br)u_{n - 1}  + atu_{n - 2} } \right)t^n }}w_n\,.
\end{equation}
We require the following identities in the sequel:
\begin{lemma}\label{lemma.t4jlec4}
The following identities hold for arbitrary $x_1$, $x_2$, $x_3$:
\begin{equation}
x_1^2  + x_2^2  + x_3^2  = (x_1 + x_2 + x_3)^2  - 2(x_1x_2 + x_1x_3 + x_2x_3)\,,
\end{equation}
\begin{equation}\label{eq.oihzhb7}
x_1^3  + x_2^3  + x_3^3  = (x_1 + x_2 + x_3)^3  - 3(x_1 + x_2 + x_3)(x_1x_2 + x_1x_3 + x_2x_3) + 3x_1x_2x_3\,,
\end{equation}
and
\begin{equation}\label{eq.jwknqg8}
\begin{split}
x_1^4  + x_2^4  + x_3^4  &= (x_1 + x_2 + x_3)^4  - 4(x_1^2  + x_2^2  + x_3^2 )(x_1x_2 + x_1x_3 + x_2x_3)\\
&\quad- 6((x_1x_2)^2  + (x_1x_3)^2  + (x_2x_3)^2 ) - 8x_1x_2x_3(x_1 + x_2 + x_3)\,.
\end{split}
\end{equation}
\end{lemma}
\begin{lemma}\label{lemma.iux34nd}
The following identity holds for integers $n$ and $m$ and arbitrary $x_1$, $x_2$, $x_3$, $c_1$, $c_2$ and $c_3$:
\[
\begin{split}
c_1x_1^{n + m}  + c_2x_2^{n + m}  + c_3x_3^{n + m}  &= (x_1^m  + x_2^m  + x_3^m )(c_1x_1^n  + c_2x_2^n  + c_3x_3^n )\\
&\quad - ((x_1 x_2 )^m  + (x_1 x_3 )^m  + (x_2 x_3 )^m )(c_1x_1^{n - m}  + c_2x_2^{n - m}  + c_3x_3^{n - m} )\\
&\qquad +  c_1x_1^{n - m} (x_2 x_3 )^m  + c_2x_2^{n - m} (x_1 x_3 )^m + c_3x_3^{n - m} (x_1 x_2 )^m\,.
\end{split}
\]
\end{lemma}
\begin{lemma}\label{lemma.pa4bxyr}
The following identity holds for integers $m$, $n$ and $k$ and arbitrary $c_1$, $c_2$, $c_3$, $x_1$, $x_2$ and $x_3$:
\[
\begin{split}
&\sum_{j = 0}^k {(c_1 x_1^{mj + n}  + c_2 x_2^{mj + n}  + c_3 x_3^{mj + n} )h^j }\\ 
&= (x_1x_2x_3)^m \frac{{c_1 x_1^{mk + n}  + c_2 x_2^{mk + n}  + c_3 x_3^{mk + n} }}{{(x_1x_2x_3)^m h^3  - ((x_1x_3)^m  + (x_2x_3)^m  + (x_1x_2)^m )h^2  + (x_1^m  + x_2^m  + x_3^m )h - 1}}h^{k + 3}\\ 
&\quad\; - \frac{(x_1^m  + x_2^m  + x_3^m )(c_1 x_1^{mk + m + n}  + c_2 x_2^{mk + m + n}  + c_3 x_3^{mk + m + n} )}{{(x_1x_2x_3)^m h^3  - ((x_1x_3)^m  + (x_2x_3)^m  + (x_1x_2)^m )h^2  + (x_1^m  + x_2^m  + x_3^m )h - 1}}h^{k + 2}\\
&\quad\; + \frac{(c_1 x_1^{mk + 2m + n}  + c_2 x_2^{mk + 2m + n}  + c_3 x_3^{mk + 2m + n} )}{{(x_1x_2x_3)^m h^3  - ((x_1x_3)^m  + (x_2x_3)^m  + (x_1x_2)^m )h^2  + (x_1^m  + x_2^m  + x_3^m )h - 1}}h^{k + 2}\\
&\quad\; + \frac{{c_1 x_1^{mk + m + n}  + c_2 x_2^{mk + m + n}  + c_3 x_3^{mk + m + n} }}{{(x_1x_2x_3)^m h^3  - ((x_1x_3)^m  + (x_2x_3)^m  + (x_1x_2)^m )h^2  + (x_1^m  + x_2^m  + x_3^m )h - 1}}h^{k + 1}\\
&\quad\;\;- (x_1x_2x_3)^m \frac{{c_1 x_1^{n - m}  + c_2 x_2^{n - m}  + c_3 x_3^{n - m} }}{{(x_1x_2x_3)^m h^3  - ((x_1x_3)^m  + (x_2x_3)^m  + (x_1x_2)^m )h^2  + (x_1^m  + x_2^m  + x_3^m )h - 1}}h^2\\ 
&\quad\;\;\;+ \frac{(x_1^m  + x_2^m  + x_3^m )(c_1 x_1^n  + c_2 x_2^n  + c_3 x_3^n ) - (c_1 x_1^{n - m}  + c_2 x_2^{n - m}  + c_3 x_3^{n - m} )}{{(x_1x_2x_3)^m h^3  - ((x_1x_3)^m  + (x_2x_3)^m  + (x_1x_2)^m )h^2  + (x_1^m  + x_2^m  + x_3^m )h - 1}}h\\
&\quad\;\;\;\;- \frac{{c_1 x_1^n  + c_2 x_2^n  + c_3 x_3^n }}{{(x_1x_2x_3)^m h^3  - ((x_1x_3)^m  + (x_2x_3)^m  + (x_1x_2)^m )h^2  + (x_1^m  + x_2^m  + x_3^m )h - 1}}\,.
\end{split}
\]
\end{lemma}
\begin{proof}
The identity expresses the linear combination of the following geometric progression summation identities:
\[
\sum_{j = 0}^k {x_1^{mj + n} h^j }  = \frac{{x_1^{mk + m + n} h^{k + 1}  - x_1^n }}{{x_1^m h - 1}},\quad\sum_{j = 0}^k {x_2^{mj + n} h^j }  = \frac{{x_2^{mk + m + n} h^{k + 1}  - x_2^n }}{{x_2^m h - 1}}
\]
and
\[
\sum_{j = 0}^k {x_3^{mj + n} h^j }  = \frac{{x_3^{mk + m + n} h^{k + 1}  - x_3^n }}{{x_3^m h - 1}}\,.
\]
\end{proof}
\begin{lemma}[{\cite[Lemma 5]{adegoke18c}}]\label{lem.h2de9i7}
Let $(X_n)$ be any arbitrary sequence, $X_n$ satisfying a four-term recurrence relation $hX_n=f_1X_{n-c_1}+f_2X_{n-c_2}+f_3X_{n-c_3}$, where $h$, $f_1$, $f_2$ and $f_3$ are arbitrary non-vanishing functions and $c_1$, $c_2$ and $c_3$ are integers. Then, the following identities hold:
\begin{equation}\label{eq.xfz9ytc}
\sum_{j = 0}^k {\sum_{i = 0}^j {\binom kj\binom jif_3^{k - j} f_2^{k + j - i} f_1^i X_{n - c_3k + (c_3 - c_2)j + (c_2 - c_1)i} } }  = h^kf_2^k X_n\,, 
\end{equation}
\begin{equation}
\sum_{j = 0}^k {\sum_{i = 0}^j {\binom kj\binom jif_2^{k - j} f_3^{k + j - i} f_1^i X_{n - c_2k + (c_2 - c_3)j + (c_3 - c_1)i} } }  = h^kf_3^k X_n\,,
\end{equation}
\begin{equation}\label{eq.vmi1kkt}
\sum_{j = 0}^k {\sum_{i = 0}^j {\binom kj\binom jif_1^{k - j} f_3^{k + j - i} f_2^i X_{n - c_1k + (c_1 - c_3)j + (c_3 - c_2)i} } }  = h^kf_3^k X_n\,, 
\end{equation}
\begin{equation}
\sum_{j = 0}^k {\sum_{i = 0}^j {( - 1)^i \binom kj\binom jih^i f_3^{k - j} f_2^{j - i} X_{n - (c_3 - c_1)k + (c_3 - c_2)j + c_2i} } }  = ( - f_1 )^k X_n \,,
\end{equation}
\begin{equation}
\sum_{j = 0}^k {\sum_{i = 0}^j {( - 1)^i \binom kj\binom jih^i f_3^{k - j} f_1^{j - i} X_{n - (c_3 - c_2)k + (c_3 - c_1)j + c_1i} } }  = ( - f_2 )^k X_n
\end{equation}
and
\begin{equation}\label{eq.o7540wl}
\sum_{j = 0}^k {\sum_{i = 0}^j {( - 1)^i \binom kj\binom jih^i f_2^{k - j} f_1^{j - i} X_{n - (c_2 - c_3)k + (c_2 - c_1)j + c_1i} } }  = ( - f_3 )^k X_n \,.
\end{equation}

\end{lemma}
\section{Main results}\label{sec.main}
\begin{theorem}[\textbf{Partial sum of the terms of the $w$ sequence with indices in arithmetic progression}]\label{theorem.fx5i6dy}
The following identity holds for integers $m$, $n$ and $k$:
\[
\begin{split}
\sum_{j = 0}^k {w_{n + jm} h^j }  &= \frac{{t^m w_{n + km} h^{k + 3}  - (v_m w_{n + m + km}  - w_{n + 2m + km} )h^{k + 2} }}{{t^m h^3  - t^m v_{ - m} h^2  + v_m h - 1}}\\
&\qquad+\frac{{w_{n + m + km} h^{k + 1}  - t^m w_{n - m} h^2  + (v_m w_n  - w_{n + m} )h - w_n }}{{t^m h^3  - t^m v_{ - m} h^2  + v_m h - 1}}\,.
\end{split}
\]

\end{theorem}
\begin{proof}
Set $(x_1,x_2,x_3)=(\alpha,\beta,\gamma)$ and $(c_1,c_2,c_3)=(A,B,C)$ in Lemma \ref{lemma.pa4bxyr}.
\end{proof}
In particular, we have
\begin{equation}
\sum_{j = 0}^k {w_j h^j }  = \frac{{tw_k h^{k + 3}  - (rw_{k + 1}  - w_{k + 2} )h^{k + 2}  + w_{k + 1} h^{k + 1}  - (c-rb-sa)h^2  + (ra - b)h - a}}{{th^3  + sh^2  + rh - 1}}
\end{equation}
and
\begin{equation}
\sum_{j = 0}^k {w_j }  = \frac{{tw_k  + (r - 1)(a + b - w_{k + 1} ) + w_{k + 2}  + as - c}}{{r + s + t - 1}}\,.
\end{equation}
Rabinowitz \cite[Identities (48), (49) and (50)]{rabinowitz96} are special cases of the identity of Theorem \ref{theorem.fx5i6dy}.
\begin{corollary}[\textbf{Generating function for the $w$ sequence with indices in arithmetic progression}]
The following identity holds for integers $m$, $n$ and $k$:
\[
\sum_{j = 0}^\infty {w_{n + jm} h^j }  =\frac{w_n - (v_m w_n  - w_{n + m} )h + t^m w_{n - m} h^2}{1 - v_m h + t^m v_{ - m} h^2 - t^m h^3}\,;
\]
and in particular,
\begin{equation}
\sum_{j = 0}^\infty {w_j h^j }  = \frac{a  - (ra - b)h + (c - rb - sa)h^2}{1 - rh -sh^2 -th^3}\,.
\end{equation}
\end{corollary}
\begin{theorem}[\textbf{Exponential generating function for the $w$ sequence with indices in arithmetic progression}]
The following identity holds for integers $m$ and $n$:
\[
\sum_{j = 0}^\infty  {\frac{{w_{mj + n} h^j }}{{j!}}}  = A\alpha^n e^{\alpha^m h}  + B\beta^n e^{\beta^m h}  + C\gamma^n e^{\gamma^m h}\,. 
\]
\end{theorem}


\begin{theorem}[\textbf{Partial sum of the squares of the terms of the $v$ sequence, with indices in arithmetic progression}]\label{theorem.ezru03x}
The following identity holds for integers $m$, $n$ and~$k$:
\[
\begin{split}
&\sum\limits_{j = 0}^k {v_{mj + n}^2 h^j }\\
&\;\;\;= \frac{{t^{2m} v_{{2n} + 2km} h^{k + 3}  - (v_{2m} v_{{2n} + {2m} + 2km}  - v_{{2n} + 4m + 2km} )h^{k + 2} }}{{t^{2m} h^3  - t^{2m} v_{ - {2m}} h^2  + v_{2m} h - 1}}\\
&\;\;\;\;\;+\frac{{v_{{2n} + {2m} + 2km} h^{k + 1}  - t^{2m} v_{{2n} - {2m}} h^2  + (v_{2m} v_{2n}  - v_{{2n} + {2m}} )h - v_{2n} }}{{t^{2m} h^3  - t^{2m} v_{ - {2m}} h^2  + v_{2m} h - 1}}\\
&\;\;\;\;\;\;+ \frac{{2t^{2m} (v_{n + km}^2  - v_{2n + 2km} )h^{k + 3} }}{{2t^{2m} h^3  - 2t^m v_m h^2  + (v_m^2  - v_{2m} )h - 2}} - \frac{{(v_m^2  - v_{2m} )(v_{n + m + km}^2  - v_{2n + 2m + 2km} )h^{k + 2} }}{{2t^{2m} h^3  - 2t^m v_m h^2  + (v_m^2  - v_{2m} )h - 2}}\\
&\;\;\;\;\;\;\;+ \frac{{2(v_{n + 2m + km}^2  - v_{2n + 4m + 2km} )h^{k + 2} }}{{2t^{2m} h^3  - 2t^m v_m h^2  + (v_m^2  - v_{2m} )h - 2}} + \frac{{2(v_{n + m + km}^2  - v_{2n + 2m + 2km} )h^{k + 1} }}{{2t^{2m} h^3  - 2t^m v_m h^2  + (v_m^2  - v_{2m} )h - 2}}\\
&\;\;\;\;\;\;\;\; - \frac{{2t^{2m} (v_{n - m}^2  - v_{2n - 2km} )h^2  + (v_n^2  - v_{2n} )(v_m^2  - v_{2m} )h}}{{2t^{2m} h^3  - 2t^m v_m h^2  + (v_m^2  - v_{2m} )h - 2}} - \frac{{2(v_{n + m}^2  - v_{2n + 2m} )h + 2(v_n^2  - v_{2n} )}}{{2t^{2m} h^3  - 2t^m v_m h^2  + (v_m^2  - v_{2m} )h - 2}}\,.
\end{split}
\]
\end{theorem}
\begin{proof}
Replace $n$ with $n+mj$ in identity \eqref{eq.xjwme1i}, multiply through by $h^j$ and sum over $j$, obtaining
\[
\sum\limits_{j = 0}^k {v_{n + mj}^2 h^j }  = \sum\limits_{j = 0}^k {v_{2n + 2mj} h^j }  + 2t^n \sum\limits_{j = 0}^k {v_{ - n - mj} (t^m h)^j }\,.
\]
Now make use of Theorem \ref{theorem.fx5i6dy} to evaluate the sums on the right hand side.
\end{proof}
\begin{corollary}[\textbf{Generating function for the squares of the terms of the $v$ sequence, with indices in arithmetic progression}]\label{corollary.ywje905}
The following identity holds for integers $m$, $n$ and $k$:
\[
\begin{split}
\sum\limits_{j = 0}^\infty {v_{mj + n}^2 h^j }&=-\frac{{t^{2m} v_{{2n} - {2m}} h^2  - (v_{2m} v_{2n}  - v_{{2n} + {2m}} )h + v_{2n} }}{{t^{2m} h^3  - t^{2m} v_{ - {2m}} h^2  + v_{2m} h - 1}}\\
&\quad- \frac{{2t^{2m} (v_{n - m}^2  - v_{2n - 2km} )h^2  + (v_n^2  - v_{2n} )(v_m^2  - v_{2m} )h}}{{2t^{2m} h^3  - 2t^m v_m h^2  + (v_m^2  - v_{2m} )h - 2}}\\
&\qquad - \frac{{2(v_{n + m}^2  - v_{2n + 2m} )h + 2(v_n^2  - v_{2n} )}}{{2t^{2m} h^3  - 2t^m v_m h^2  + (v_m^2  - v_{2m} )h - 2}}\,.
\end{split}
\]
\end{corollary}
\begin{theorem}[\textbf{Partial sum of the products of the terms of the $u$ sequence and the $v$ sequence, with indices in arithmetic progression}]\label{theorem.fkqnchf}
The following identity holds for integers $m$, $n$:
\[
\begin{split}
&\sum_{j = 0}^k {u_{n+mj} v_{n+mj}h^j }\\
&\;\;\;= \frac{{t^{2m} u_{{2n} + 2km} h^{k + 3}  - (v_{2m} u_{{2n} + {2m} + 2km}  - u_{{2n} + 4m + 2km} )h^{k + 2} }}{{t^{2m} h^3  - t^{2m} v_{ - {2m}} h^2  + v_{2m} h - 1}}\\
&\;\;\;\;\;+\frac{{u_{{2n} + {2m} + 2km} h^{k + 1}  - t^{2m} u_{{2n} - {2m}} h^2  + (v_{2m} u_{2n}  - u_{{2n} + {2m}} )h - u_{2n} }}{{t^{2m} h^3  - t^{2m} v_{ - {2m}} h^2  + v_{2m} h - 1}}\\
&\;\;\;\;\;\;- \frac{{t^{2m} (u_{2n + 2mk}  - u_{n + mk} v_{n + mk} )h^{k + 3} }}{{t^{2m} h^3  - t^m v_m h^2  + t^m v_{ - m} h - 1}} + \frac{{t^m v_{ - m} (u_{2n + 2m + 2mk}  - u_{n + m + mk} v_{n + m + mk} )h^{k + 2} }}{{t^{2m} h^3  - t^m v_m h^2  + t^m v_{ - m} h - 1}}\\
&\;\;\;\;\;\;\;- \frac{{(u_{2n + 4m + 2mk}  - u_{n + 2m + mk} v_{n + 2m + mk} )h^{k + 2} }}{{t^{2m} h^3  - t^m v_m h^2  + t^m v_{ - m} h - 1}} - \frac{{(u_{2n + 2m + 2mk}  - u_{n + m + mk} v_{n + m + mk} )h^{k + 1} }}{{t^{2m} h^3  - t^m v_m h^2  + t^m v_{ - m} h - 1}}\\
&\;\;\;\;\;\;\;\;+ \frac{{t^{2m} (u_{2n - 2m}  - u_{n - m} v_{n - m} )h^2 }}{{t^{2m} h^3  - t^m v_m h^2  + t^m v_{ - m} h - 1}} - \frac{{(t^m v_{ - m} (u_{2n}  - u_n v_n ) - u_{2n + 2m}  + u_{n + m} v_{n + m} )h}}{{t^{2m} h^3  - t^m v_m h^2  + t^m v_{ - m} h - 1}}\\
&\;\;\;\;\;\;\;\;\;+ \frac{{u_{2n}  - u_n v_n }}{{t^{2m} h^3  - t^m v_m h^2  + t^m v_{ - m} h - 1}}
\end{split}
\]
\end{theorem}
\begin{proof}
Replace $n$ with $n+mj$ in identity \eqref{eq.gih8b7c}, multiply through by $h^j$ and sum over $j$, obtaining
\[
\sum\limits_{j = 0}^k {u_{n + mj} v_{n + mj} h^j }  = \sum\limits_{j = 0}^k {u_{2n + 2mj} h^j }  - t^n \sum\limits_{j = 0}^k {u_{ - n - mj} (t^m h)^j }\,.
\]
Now make use of Theorem \ref{theorem.fx5i6dy} to evaluate the sums on the right hand side.
\end{proof}
\begin{corollary}[\textbf{Generating function for the products of the terms of the $u$ sequence and the $v$ sequence, with indices in arithmetic progression}]
The following identity holds for integers $m$, $n$:
\[
\begin{split}
&\sum_{j = 0}^\infty {u_{n+mj} v_{n+mj}h^j }\\
&\;\;\;\;\;=-\frac{{t^{2m} u_{{2n} - {2m}} h^2  - (v_{2m} u_{2n}  - u_{{2n} + {2m}} )h + u_{2n} }}{{t^{2m} h^3  - t^{2m} v_{ - {2m}} h^2  + v_{2m} h - 1}}\\
&\;\;\;\;\;\;\;\;+ \frac{{t^{2m} (u_{2n - 2m}  - u_{n - m} v_{n - m} )h^2 }}{{t^{2m} h^3  - t^m v_m h^2  + t^m v_{ - m} h - 1}} - \frac{{(t^m v_{ - m} (u_{2n}  - u_n v_n ) - u_{2n + 2m}  + u_{n + m} v_{n + m} )h}}{{t^{2m} h^3  - t^m v_m h^2  + t^m v_{ - m} h - 1}}\\
&\;\;\;\;\;\;\;\;\;+ \frac{{u_{2n}  - u_n v_n }}{{t^{2m} h^3  - t^m v_m h^2  + t^m v_{ - m} h - 1}}
\end{split}
\]
\end{corollary}
\begin{theorem}[\textbf{Multiple argument formulas for the $v$ sequence}]
The following identities hold for integer $n$:
\begin{equation}\label{eq.v6amjy5}
v_{2n}=u_nv_{n+1}+(u_{n+1}-ru_n)v_n+tu_{n-1}v_{n-1}\,,
\end{equation}

\begin{equation}\label{eq.zthm4ml}
2v_{3n}  = 6t^n  - v_n^3  + 3u_n v_nv_{n+1}  + 3(u_{n + 1}  - ru_n )v_n^2  + 3tu_{n - 1} v_{n - 1} v_n\,,
\end{equation}

\begin{equation}\label{eq.bo6iltm}
\begin{split}
2v_{4n}  &= 2(u_{n + 1}  - ru_n )v_n^3 + (2tu_{n - 1} v_{n - 1}  + (u_{n + 1}  - ru_n )^2  + 2u_n v_{n + 1} )v_n^2\\
&\quad+ 2((u_{n + 1}  - ru_n )u_n v_{n + 1}  + 4t^n  + tu_{n - 1} v_{n - 1} (u_{n + 1}  - ru_n ))v_n\\
&\qquad+ t^2 u_{n - 1}^2 v_{n - 1}^2  + 2tu_n v_{n + 1} u_{n - 1} v_{n - 1}  + u_n^2 v_{n + 1}^2  - v_n^4\,.
\end{split}
\end{equation}
\end{theorem}
\begin{proof}
Identity \eqref{eq.v6amjy5} is obtained by setting $m=n$ in the identity of Theorem \ref{theorem.jbqvo05} with $(w_n)=(v_n)$. Identities \eqref{eq.zthm4ml} and \eqref{eq.bo6iltm} come from Lemma \ref{lemma.t4jlec4}, identities \eqref{eq.oihzhb7} and \eqref{eq.jwknqg8} respectively, by setting $(x_1,x_2,x_3)=(\alpha,\beta,\gamma)$.
\end{proof}
\begin{theorem}[\textbf{An identity for the $u$ sequence}]
The following identity holds for integer~$n$:
\[
\begin{split}
t^n  &= t^3 u_{n - 2}^3  + t(u_n  - ru_{n - 1} )^3  + t^2 u_{n - 1}^3 \\
&\quad + tu_{n - 1} (u_n  - ru_{n - 1} )( - su_{n - 1}  + r(u_n  - ru_{n - 1} ))\\
&\quad\; + ru_{n - 2} (u_n  - ru_{n - 1} )(tu_{n - 1}  + u_{n - 2} )\\
&\quad\;\; + t^2 (r^2  + 2s)u_{n - 1} u_{n - 2}^2  - t(r^3  + 3rs + 3t)u_{n - 1} u_{n - 2} (u_n  - ru_{n - 1} )\\
&\quad\;\;\; - su_{n - 2} (u_n  - ru_{n - 1} )^2  + t(s^2  - 2rt)u_{n - 2} u_{n - 1}^2\,.
\end{split}
\]
\end{theorem}
\begin{proof}
This identity comes from multiplying out $\alpha^n\beta^n\gamma^n$, using identity \eqref{eq.txorvr0} and the corresponding $\beta^n$ and $\gamma^n$ identities.
\end{proof}
\begin{theorem}[\textbf{An identity for the $v$ sequence}]
The following identity holds for integer~$n$:
\[
t^{ - n} (v_n^3 - v_{3n}) = t^n (v_{ - n}^3 - v_{ - 3n})\,.
\]
\end{theorem}
\begin{proof}
Setting $(x_1,x_2,x_3)=(\alpha,\beta,\gamma)$ in identity \eqref{eq.oihzhb7}, Lemma \ref{lemma.t4jlec4} gives
\begin{equation}
t^{ - n} v_{3n}  = t^{ - n} v_n^3  - 3v_n v_{ - n}  + 3\,,
\end{equation}
which, by replacing $n$ with $-n$ also means
\begin{equation}
t^n v_{ - 3n}  = t^n v_{ - n}^3  - 3v_n v_{ - n}  + 3\,,
\end{equation}
and hence the identity of the theorem.
\end{proof}
\begin{theorem}[\textbf{An identity for the $w$ sequence}]\label{theorem.b2dk1c1}
The following identity holds for integers $n$ and $m$:
\[\tag{H}
w_{n + m}  = v_m w_n  - t^m v_{ - m} w_{n - m}  + t^m w_{n - 2m}\,,
\]
or, by employing identity \eqref{eq.xjwme1i},
\begin{equation}\label{eq.xuxbcyk}
2w_{n + m}  = 2v_m w_n  - (v_m^2  - v_{2m} )w_{n - m}  + 2t^m w_{n - 2m}\,.
\end{equation}
\end{theorem}
\begin{proof}
Set $(x_1,x_2,x_3)=(\alpha,\beta,\gamma)$ and $(c_1,c_2,c_3)=(A,B,C)$ in Lemma \ref{lemma.iux34nd}.
\end{proof}
Identity (H) was discovered by Howard \cite{howard01}. However, the proof here is much simpler than the one given by Howard.

\medskip

Setting $m=n$ in identity \eqref{eq.xuxbcyk} gives
\begin{equation}\label{eq.gza79eq}
w_{ - n}  = \frac{{2w_{2n}  - 2v_n w_n  + a(v_n^2  - v_{2n} )}}{{2t^n }}\,,
\end{equation}
which is a simpler alternative to identity \eqref{eq.fsxsk5o} for the direct access to negative index terms of the $w$ sequence.

\medskip

From identity (H) we get
\begin{equation}\label{eq.vp79n99}
u_{n + m}  = v_m u_n  - t^m v_{ - m} u_{n - m}  + t^m u_{n - 2m}\,,
\end{equation}
which gives
\begin{equation}\label{eq.gih8b7c}
u_{ - n}  = (u_{2n}  - u_n v_n )/t^n 
\end{equation}
at $m=n$  and
\begin{equation}
u_{-n}=(u_{2n-3}-u_{n-2}v_{n-1})/t^{n-1}
\end{equation}
at $m=n+1$.

\medskip

Setting $n=2m$, and $n=2m-1$, in turn, in \eqref{eq.vp79n99} and changing index back to $n$,  we have
\begin{equation}
v_{ - n}  = \frac{{u_{2n} v_n  - u_{3n} }}{{t^n u_n }},\quad n\notin\{-1,0\}\,,
\end{equation}
and
\begin{equation}
v_{ - n}  = \frac{{u_{2n - 1} v_n  - u_{3n - 1} }}{{t^n u_{n - 1} }},\quad n\notin\{0,1\} \,.
\end{equation}

\begin{theorem}[\textbf{Reciprocal sums involving the $u$ sequence}]
The following identities hold for integers $k$ and $n$ for which the summand is non-singular in the summation interval:
\begin{equation}\label{eq.xr55z5q}
u_{n + 1} u_{n - k} \sum_{j = 0}^k {\frac{t^ju_{k - n - j - 1} }{{u_{n - k + j} u_{n - k + j + 1} }}}  = t^{k - n} (u_n u_{n - k}  - u_{n + 1} u_{n - k - 1} )\,,
\end{equation}
\begin{equation}
u_n u_{n - k - 1} \sum_{j = 0}^k {\frac{{t^j u_{k - n - j - 1} }}{{u_{n - k + j} u_{n - k + j - 1} }}}  = t^{k - n} (u_n u_{n - k}  - u_{n + 1} u_{n - k - 1} )\,.
\end{equation}
\end{theorem}
\begin{proof}
Division through identity \eqref{eq.lmwxb7q} by $u_nu_{n-1}$ and a shift of index gives
\begin{equation}
\frac{{t^n u_{ - n - 1} }}{{u_n u_{n + 1} }} = \frac{{u_n }}{{u_{n + 1} }} - \frac{{u_{n - 1} }}{{u_n }}\,,
\end{equation}
from which identity \eqref{eq.xr55z5q} follows by telescoping summation.
\end{proof}
\begin{theorem}[\textbf{Double binomial summation identities invoked by the definition of the generalized third order sequence}]
The following identities hold for positive integer $k$ and integers $m$ and $n$:
\begin{equation}
\sum_{j = 0}^k {\sum_{i = 0}^j {\binom kj\binom jit^{k - j} s^{k + j - i} r^i w_{n - 3k + j + i} } }  = s^k w_n\,,
\end{equation}
\begin{equation}
\sum_{j = 0}^k {\sum_{i = 0}^j {\binom kj\binom jis^{k - j} t^{k + j - i} r^i w_{n - 2k - j + 2i} } }  = t^k w_n\,,
\end{equation}
\begin{equation}
\sum_{j = 0}^k {\sum_{i = 0}^j {\binom kj\binom jir^{k - j} t^{k + j - i} s^i w_{n - k - 2j + i} } }  = t^k w_n\,,
\end{equation}
\begin{equation}
\sum_{j = 0}^k {\sum_{i = 0}^j {( - 1)^i \binom kj\binom jit^{k - j} s^{j - i}w_{n - 2k + j + 2i} } }  = ( - r)^k w_n\,,
\end{equation}
\begin{equation}
\sum_{j = 0}^k {\sum_{i = 0}^j {( - 1)^i \binom kj\binom jit^{k - j} r^{j - i} w_{n - k + 2j + i} } }  = ( - s)^k w_n\,,
\end{equation}
\begin{equation}
\sum_{j = 0}^k {\sum_{i = 0}^j {( - 1)^i \binom kj\binom jis^{k - j} r^{j - i} w_{n + k + j + i} } }  = ( - t)^k w_n\,.
\end{equation}
\end{theorem}
\begin{proof}
The identities follow directly from Lemma \ref{lem.h2de9i7} and the recurrence relation \eqref{eq.vhrb5b3}.
\end{proof}
\begin{theorem}[\textbf{More double binomial summation identities involving third order sequences}]
The following identities hold for positive integer $k$ and integers $m$ and $n$:
\begin{equation}
\sum_{j = 0}^k {\sum_{i = 0}^j {( - 1)^i \binom kj\binom ji(tu_{m - 1} )^{k - j}u_m^{k + j - i} w_{n - k + 2j + (m - 1)i} } }  = (u_{m + 1}  - ru_m )^k ( - u_m )^k w_n\,,
\end{equation}
\begin{equation}
\sum_{j = 0}^k {\sum_{i = 0}^j {( - 1)^i \binom kj\binom jiu_m^{k - j} (tu_{m - 1} )^{k + j - i} w_{n + k - 2j + (m + 1)i} } }  = (u_{m + 1}  - ru_m )^k ( - tu_{m - 1} )^k w_n\,,
\end{equation}
\begin{equation}
\sum_{j = 0}^k {\sum_{i = 0}^j {( - 1)^j \binom kj\binom ji(tu_{m - 1} )^{k + j - i} u_m^i w_{n + mk - (m + 1)j + 2i} } }  = (u_{m + 1}  - ru_m )^k (tu_{m - 1} )^k w_n\,,
\end{equation}
\begin{equation}
\sum_{j = 0}^k {\sum_{i = 0}^j {\binom kj\binom ji(tu_{m - 1} )^{k - j} u_m^{j - i} (u_{m + 1}  - ru_m )^i w_{n - (m + 1)k + 2j - i} } }  = w_n\,,
\end{equation}
\begin{equation}
\sum_{j = 0}^k {\sum_{i = 0}^j {( - 1)^{i + j} \binom kj\binom ji(tu_{m - 1} )^{k - j} (u_{m + 1}  - ru_m )^i w_{{n - 2k + (m + 1)j - mi}} } }  = ( - 1)^k u_m^k w_n
\end{equation}
and
\begin{equation}
\sum_{j = 0}^k {\sum_{i = 0}^j {( - 1)^{i + j} \binom kj\binom jiu_m ^{k - j} (u_{m + 1}  - ru_m )^i w_{n + 2k + (m - 1)j - mi} } }  = ( - 1)^k (tu_{m - 1} )^k w_n\,.
\end{equation}
\end{theorem}
\begin{proof}
Write the identity of Theorem \ref{theorem.jbqvo05} as
\[
(u_{m + 1}  - ru_m )w_n  = w_{n + m}  - u_m w_{n + 1}  - tu_{m - 1} w_{n - 1}\,. 
\]
Identify $h=u_{m+1}-ru_m$, $f_1=1$, $f_2=-u_{m}$, $f_3=tu_{m-1}$, $c_1=-m$, $c_2=-1$ and $c_3=1$ and use these in Lemma \ref{lem.h2de9i7}.
\end{proof}

\hrule

\noindent 2010 {\it Mathematics Subject Classification}:
Primary 11B39; Secondary 11B37.

\noindent \emph{Keywords: }
Third order sequence, Tribonacci sequence, Padovan sequence, Perrin sequence, summation identity, partial sum, generating function.

\hrule



\end{document}